\documentclass[12pt,oneside]{amsart}
\usepackage[margin=1in]{geometry}
\usepackage{enumitem,verbatim}
\usepackage{microtype}
\usepackage{physics}
\usepackage{hyperref}
\usepackage{color}
\hypersetup{colorlinks,linkcolor=blue,citecolor=cyan}
\usepackage{amsthm}
\usepackage[nameinlink]{cleveref}
\usepackage{amsmath,amssymb}
\usepackage{tikz}
\usepackage{longtable}
\usepackage{bm}
\usepackage{subcaption}

\makeatletter
\newcommand{\newreptheorem}[2]{\newtheorem*{rep@#1}{\rep@title}
\newenvironment{rep#1}[1]{\def\rep@title{#2 \ref*{##1}}\begin{rep@#1}}{\end{rep@#1}}}
\makeatother

\theoremstyle{definition}

\newtheorem{thm}{Theorem}
\newtheorem{cor}{Corollary}

\newtheorem{prop}{Proposition}
\newtheorem{conj}{Conjecture}

\newtheorem{rmk}{Remark}
\newreptheorem{defn}{Definition}
\newreptheorem{thm}{Theorem}
\newreptheorem{cor}{Corollary}
\newreptheorem{conj}{Conjecture}
\newcommand{\qbin}[2]{\begin{bmatrix}{#1}\\ {#2}\end{bmatrix}}

\title{Large color $R$-matrix for knot complements and strange identities}

\author{Sunghyuk Park}

\address{Division of Physics, Mathematics and Astronomy,
California Institute of Technology,
1200 E. California Blvd.,
Pasadena,
CA 91125}
\email{spark3@caltech.edu}

\thanks{}

\begin{document}
\maketitle
\begin{abstract}
    The Gukov-Manolescu series, denoted by $F_K$, is a conjectural invariant of knot complements that, in a sense, analytically continues the colored Jones polynomials. In this paper we use the large color $R$-matrix to study $F_K$ for some simple links. Specifically, we give a definition of $F_K$ for positive braid knots, and compute $F_K$ for various knots and links. As a corollary, we present a class of `strange identities' for positive braid knots.  
\end{abstract}
\tableofcontents

\section{Introduction}
Recently, Gukov and Manolescu \cite{GM} conjectured the existence of a two-variable series $F_K(x,q)$ (let's call it the \emph{GM series}) for every knot $K$, which is in a sense an ``analytically continued'' version of the colored Jones polynomials: 
\begin{conj}[Conjecture 1.5 in \cite{GM}]\label{FKconj}
    There is a knot invariant $F_K(x,q)$, a series in two variables $x$ and $q$ with integer coefficients assigned to each knot $K$, such that
    \begin{equation}
        \frac{F_K(x,q=e^\hbar)}{x^{\frac{1}{2}}-x^{-\frac{1}{2}}} = \sum_{j\geq 0}\frac{P_K(j;x)}{\Delta_K(x)^{2j+1}}\frac{\hbar^j}{j!},
    \end{equation}
    where the r.h.s.\ is the \emph{large color expansion} (i.e.\ Melvin-Morton-Rozansky expansion \cite{MM, BG, R1}) of the reduced colored Jones polynomials $J_K(n;q=e^\hbar)$ expanded near $\hbar=0$ while keeping $x=q^n=e^{n\hbar}$ fixed,\footnote{In our convention, $n$ corresponds to the $n$-dimensional representation of $\mathfrak{sl}_2$. $J_K(n;q)$ are \emph{reduced} colored Jones polynomials which are $1$ for the unknot, as opposed to \emph{unreduced} colored Jones polynomials $\Tilde{J}_K(n;q)$ which are $[n]$ for the unknot. $P_K(j;x)\in \mathbb{Z}[x^{\pm 1}]$ are some polynomials determined in this expansion.} and $\Delta_K$ is the Alexander polynomial of $K$. 
\end{conj}
\begin{conj}[Conjecture 1.6 in \cite{GM}]\label{quantumA}
    Moreover, if $\hat{A}_K(\hat{x},\hat{y},q)$ is the quantum $A$-polynomial \cite{Gar, Guk} annihilating the unreduced colored Jones polynomials, then 
    \begin{equation}
        \hat{A}_K(\hat{x},\hat{y},q) F_K(x,q) = 0.
    \end{equation}
\end{conj}

Their conjectures were motivated by yet another conjecture due to Gukov, Pei, Putrov and Vafa \cite{GPV, GPPV} that there exists an invariant of $3$-manifolds $\hat{Z}$ (let's call it the \emph{GPPV series}) valued in $q$-series with integer coefficients. Conjecturally, the GPPV series are categorifiable (i.e.\ they are graded Euler characteristics of some homology theory), and their categorification should be thought of as $3$-manifold analog of the Khovanov homology \cite{Kh}.\footnote{This fascinating conjecture is one of our main motivations. However, almost nothing is known about the categorification of these $q$-series invariants, and in this article we focus on $q$-series invariants, not their categorification.}
The GM series should be thought of as the GPPV series for knot complements. In fact, Gukov and Manolescu conjectured that there is a surgery formula relating the GM series of a knot complement and the GPPV series for the 3-manifold obtained by a Dehn surgery on the knot:
\begin{conj}[Conjecture 1.7 in \cite{GM}]\label{conj:surgery}
    Let $S^3_{\frac{p}{r}}(K)$ be the 3-manifold obtained by $\frac{p}{r}$-surgery on $K\subset S^3$, and let $\mathcal{L}^{b}_{\frac{p}{r}}$ be the `Laplace transform' defined by
    \begin{equation}
        \mathcal{L}^{b}_{\frac{p}{r}}: x^u \mapsto \begin{cases}q^{-\frac{r}{p}u^2}&\text{ if }u\in \frac{b}{r}+\frac{p}{r}\mathbb{Z}\\ 0&\text{ otherwise}\end{cases}.
    \end{equation}
    Then
    \begin{equation}\label{knotsurgery}
        \hat{Z}_b(S^3_{\frac{p}{r}}(K)) = \epsilon q^d \mathcal{L}^{b}_{\frac{p}{r}}\qty[(x^{\frac{1}{2r}}-x^{-\frac{1}{2r}})F_K(x,q)]
    \end{equation}
    for some sign $\epsilon\in \{\pm 1\}$ and $d\in \mathbb{Q}$, whenever the r.h.s.\ makes sense. 
\end{conj}
Unfortunately the r.h.s.\ of \eqref{knotsurgery} doesn't always make sense; in practice, it only makes sense when $-\frac{r}{p}$ is big enough to make the r.h.s. a Laurent power series in $q$. 
Therefore, the problem of mathematically defining the GPPV series could be viewed as a two-step problem: 
\begin{enumerate}
    \item Give a mathematical definition of $F_K$ for links. 
    \item Make a surgery formula that works for all surgery coefficients. 
\end{enumerate}
In this paper we make some progress toward the first step, and argue that the first step can actually be reduced to the second step.

\addtocontents{toc}{\protect\setcounter{tocdepth}{1}}
\subsection*{Statement of the results}
The purpose of this paper is to introduce the large color $R$-matrix as a useful tool to compute $F_K$ for a much bigger class of knots than it was previously known. (Previously, $F_K$ was only computed for torus knots and, experimentally, for the figure-eight knot \cite{GM}.)

Our main theorem is the following. 
\begin{thm}\label{pbthm}
    Conjecture \ref{FKconj} is true for positive braid knots, and the corresponding $F_K$ can be computed using the large color $R$-matrix. 
\end{thm}
A nice corollary of Theorem \ref{pbthm} is a new class of `strange identities' for positive braid knots, which has the Kontsevich-Zagier strange identity \cite{Z} as a special case for the right-handed trefoil knot. 
\begin{cor}\label{strangeidentity}
    For every positive braid knot $K$, there is a strange identity
    \begin{equation}
        \langle K \rangle \text{``}=\text{''} \frac{F_K^{\pm}(x,q)}{x^{\frac{1}{2}}-x^{-\frac{1}{2}}}\bigg\vert_{x=1},
    \end{equation}
    where the l.h.s.\ is the Kashaev's invariant defined by $\langle K \rangle_n = J_K(n;q=e^{\frac{2\pi i}{n}})$ at each root of unity $q = e^{\frac{2\pi i}{n}}$ and the r.h.s.\ is a $q$-series defined inside the unit disk. The quoted equality means that if we set $q=e^\hbar$, their perturbative expansions in $\hbar$ agree. 
\end{cor}

Although our main theorem is on positive braid knots only, it turns out that the method of large color $R$-matrix is more robust. Indeed, using the large color $R$-matrix, we experimentally compute $F_K$ for many knots and links beyond positive braids, including positive double twist knots, fibered strongly quasi-positive knots up to 10 crossings, and the Whitehead link. 

Based on our examples of $F_K$ for some links, we extend the Conjectures \ref{FKconj} and \ref{conj:surgery} to links. These conjectures will be illustrated through various non-trivial examples. 
\begin{conj}
    There is a link invariant $F_L(x_1,\cdots,x_l,q)$, a series in $x_1,\cdots,x_l$ and $q$ with integer coefficients, where $l$ is the number of components of the link $L$, such that
    \begin{equation}
        F_L(x_1,\cdots,x_l,q=e^\hbar) = \sum_{j\geq 0}\frac{P_L(j;x_1,\cdots,x_l)}{\nabla_L(x_1,\cdots,x_l)^{2j+1}}\frac{\hbar^j}{j!},
    \end{equation}
    where the r.h.s.\ is the large color expansion \cite{R2} of the colored Jones polynomials $J_L(n_1,\cdots,n_l;q=e^\hbar)$ expanded near $\hbar = 0$ while keeping $x_i = q^{n_i} = e^{n_i \hbar}$ fixed, for each $1\leq i\leq l$, and $\nabla_L$ is the Alexander-Conway function of $L$. 
    
    Moreover, $F_L$ is annihilated by the quantum $A$-ideal annihilating the colored Jones polynomials of $L$. That is, it is annihilated by any $q$-difference equation that annihilates the colored Jones polynomials of $L$. 
\end{conj}
\begin{conj}\label{conj:linksurg}
    Let $S^3_{p_1,\cdots,p_l}(L)$ be the 3-manifold obtained by $(p_1,\cdots,p_l)$ surgery on $L\subset S^3$, and let $\mathrm{B}$ be the $l\times l$ linking matrix defined by
    \begin{equation}
        \mathrm{B}_{ij} = \begin{cases}p_i &\text{ if }i = j\\ lk(i,j) &\text{ otherwise}\end{cases}.
    \end{equation}
    Let $\mathcal{L}^{b}_{\mathrm{B}}$ be the `Laplace transform' defined by
    \begin{equation}
        \mathcal{L}^{b}_{\mathrm{B}}: x^u \mapsto \begin{cases}q^{-(u,\mathrm{B}^{-1}u)}&\text{ if }u\in b+\mathrm{B}\mathbb{Z}^l\\ 0&\text{ otherwise}\end{cases}.
    \end{equation}
    Then
    \begin{equation}
        \hat{Z}_b(S^3_{p_1,\cdots,p_l}(L)) = \epsilon q^d \mathcal{L}_{\mathrm{B}}^b\qty[(x_1^{\frac{1}{2}}-x_1^{-\frac{1}{2}})\cdots (x_l^{\frac{1}{2}}-x_l^{-\frac{1}{2}})F_L(x_1,\cdots,x_l,q)]
    \end{equation}
    for some sign $\epsilon \in \{\pm 1\}$ and $d\in \mathbb{Q}$, whenever the r.h.s.\ makes sense. 
\end{conj}

One takeaway from this paper is that we provide an entire playground of new examples to play with and test conjectures. (For example, previously the only hyperbolic knot for which $F_K$ was computed was the figure-eight knot in \cite{GM}, but now we have infinitely more!)
We believe our result is an important step toward the eventual mathematical definition of these conjectural $3$-manifold invariants, which we hope to see soon.

\subsection*{Organization of the paper}
After giving a brief review of $U_q(\mathfrak{sl}_2)$ and the $R$ matrix on finite-dimensional representations in Section \ref{sec:review}, we describe the large color $R$ matrix on Verma modules in Section \ref{sec:largecolorR}. 

Then, in Section \ref{sec:main}, we prove Theorem \ref{pbthm}, and use the large color $R$-matrix to compute $F_K$ in various examples. 

In Section \ref{sec:surgeries}, we perform various non-trivial consistency checks of the surgery formula (Conjectures \ref{conj:surgery} and \ref{conj:linksurg}), and demonstrate how to reverse-engineer $F_K$ for various links. 

Finally, in Section \ref{sec:strangeid}, we explain how our result can be viewed as a generalization of Kontsevich-Zagier's strange identity.

\subsection*{Acknowledgments}
I would like to thank Sergei Gukov and Ciprian Manolescu for insightful discussions, as well as Piotr Kucharski, Robert Osburn, and Nikita Sopenko for useful conversations. 
The author was supported by Kwanjeong Educational Foundation.

\addtocontents{toc}{\protect\setcounter{tocdepth}{2}}
\section{Review of $U_q(\mathfrak{sl}_2)$ and $R$-matrix}\label{sec:review}
Here we give a brief review the quantum enveloping algebra of $\mathfrak{sl}_2$ and the $R$ matrix on its finite dimensional representations. These are standard materials, and for more on this topic, see e.g.\ \cite{Majid, CP}. 
\subsection{$U_q(\mathfrak{sl}_2)$}
\subsubsection{Classical $\mathfrak{sl}_2$}
We use the standard notations: 
\[e = \begin{pmatrix}0 & 1 \\ 0 & 0\end{pmatrix},\quad f = \begin{pmatrix}0 & 0 \\ 1 & 0\end{pmatrix},\quad h = \begin{pmatrix}1 & 0 \\ 0 & -1\end{pmatrix}\]
\[[h,e]=2e,\quad [h,f]=-2f,\quad [e,f]=h.\]
Let's denote by $V_n$ the irreducible $n$-dimensional representation of $\mathfrak{sl}_2$. Then 
\[V_n = \bigoplus_{j\in \mathbb{Z}}V_n(j)\]
where $V_n(j)$ is the $j$-eigenspace of $V_n$ under the action by $h$. The action of $e$ and $f$ can be drawn diagrammatically as
\[V_n(1-n)\overset{e}{\underset{f}{\rightleftharpoons}}\cdots\overset{e}{\underset{f}{\rightleftharpoons}}V_n(j)\overset{e}{\underset{f}{\rightleftharpoons}} V_n(j+2)\overset{e}{\underset{f}{\rightleftharpoons}}\cdots\overset{e}{\underset{f}{\rightleftharpoons}}V_n(n-1),\]
where \[[e,f]v = jv\] for $v\in V_n(j)$. 

Every $V_n$ appears as a factor of tensor powers of $V_2$: 
\[V_n = \Delta_{n-1}(V_2),\]
where $\Delta_n(x)$ is the $n$-th Chebyshev polynomial defined inductively by
\[\Delta_0 = 1,\quad \Delta_1(x) = x,\quad \Delta_{n+1}(x) = x\Delta_n(x) - \Delta_{n-1}(x).\]
\subsubsection{Quantum $\mathfrak{sl}_2$}
$U_q(\mathfrak{sl}_2)$ is the associative algebra over $\Bbbk := \mathbb{C}(q^{\frac{1}{2}})$ with generators $e,f,q^{\frac{h}{2}}$ and relations
\begin{align*}
    q^{i\frac{h}{2}}q^{j\frac{h}{2}} &= q^{(i+j)\frac{h}{2}},\quad q^{0\frac{h}{2}}=1,\\
    q^{\frac{h}{2}} e q^{-\frac{h}{2}} &= q\, e,\quad q^{\frac{h}{2}} f q^{-\frac{h}{2}} = q^{-1} f,\\
    [e,f] &= \frac{q^{\frac{h}{2}}-q^{-\frac{h}{2}}}{q^{\frac{1}{2}}-q^{-\frac{1}{2}}}.
\end{align*}
Observe that formally in the classical limit $q\rightarrow 1$, these relations reduces to the classical relations. 

Let's denote by $V_n$ the irreducible $n$-dimensional representation of $U_q(\mathfrak{sl}_2)$. Then
\[V_n = \bigoplus_{j\in \mathbb{Z}}V_n(j),\]
where $V_n(j)$ is the $q^{\frac{j}{2}}$-eigenspace of $V_n$ under the action by $q^{\frac{h}{2}}$. 
Diagrammatically we have the same picture as before
\[V_n(1-n)\overset{e}{\underset{f}{\rightleftharpoons}}\cdots\overset{e}{\underset{f}{\rightleftharpoons}}V_n(j)\overset{e}{\underset{f}{\rightleftharpoons}} V_n(j+2)\overset{e}{\underset{f}{\rightleftharpoons}}\cdots\overset{e}{\underset{f}{\rightleftharpoons}}V_n(n-1)\]
except that now 
\[[e,f]v = [j]\,v\]
for $v\in V_n(j)$ and $[j] := \frac{q^{\frac{j}{2}}-q^{-\frac{j}{2}}}{q^{\frac{1}{2}}-q^{-\frac{1}{2}}}$. 

The quantum dimension of $V_n$ is
\[\dim_q V_n = \Tr q^{\frac{h}{2}} = q^{\frac{n-1}{2}} + q^{\frac{n-3}{2}} + \cdots + q^{\frac{1-n}{2}} = [n] = \Delta_{n-1}([2]).\]

\subsection{$R$-matrix for $V_n\otimes V_m$}\label{subsec:finRmatrix}
Let $V_n$ be the $n$-dimensional $U_q(\mathfrak{sl}_2)$-module as before. We choose the basis vectors $\{v_j\}$, $j = 1-n, 3-n,\cdots,n-1$ of $V_n$ such that 
\begin{align*}
    ev_j &= [\frac{n-1-j}{2}]\,v_{j+2},\\
    fv_j &= [\frac{n-1+j}{2}]\,v_{j-2},\\
    q^{\frac{h}{2}}v_j &= q^{\frac{j}{2}}v_j.
\end{align*}
It is easy to check that this is indeed well-defined: 
\[[e,f]v_j = \qty([\frac{n-1-(j-2)}{2}][\frac{n-1+j}{2}]-[\frac{n-1+(j+2)}{2}][\frac{n-1-j}{2}])v_j = [j]\, v_j,\]
which follows from the identity
\[[a][b]-[a+j][b-j] = [j][a-b+j].\]

In this basis, the $R$-matrix on $V_n\otimes V_n$ has the following form (see \cite{KM, R1}):\footnote{Here $[k]! := [1][2]\cdots [k]$ is the $q$-fatorial, and $\qbin{n}{k}:=\frac{[n]!}{[k]![n-k]!}$ is the $q$-binomial coefficient.} 
\begin{equation}
R(v_i\otimes v_j) = \sum_{k\geq 0} (q^{\frac{1}{2}}-q^{-\frac{1}{2}})^k[k]! \qbin{\frac{n-1-i}{2}}{k}\qbin{\frac{n-1+j}{2}}{k}q^{\frac{1}{4}\qty(ij -k(i-j) -k(k+1))} v_{i+2k}\otimes v_{j-2k},
\end{equation}
where the range of summation is $0\leq k\leq \min\{\frac{n-1-i}{2},\frac{n-1+j}{2}\}$. 
The flipped $R$-matrix is the $R$-matrix followed by the permutation operator $P: v_i\otimes v_j \mapsto v_j\otimes v_i$: 
\[\check{R} = PR.\]
It satisfies the quantum Yang-Baxter equation: 
\begin{equation}
    \check{R}_{23}\check{R}_{12}\check{R}_{23} = \check{R}_{12}\check{R}_{23}\check{R}_{12}.
\end{equation}
Therefore it induces a representation of the braid group. 
\begin{figure}[h]
    \centering
    \includegraphics[scale=0.5]{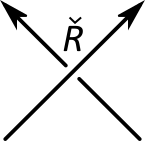}
    \caption{Graphically, $\check{R}$ corresponds to the positive crossing, and $\check{R}^{-1}$ the negative crossing.}
    \label{fig:Rmatrix}
\end{figure}

The \emph{unreduced} $V_n$-colored Jones polynomial of a knot $K$ presented as the closure of a braid $\beta$ is given by
\begin{equation}
    \Tilde{J}_K(n;q) = q^{-\frac{n^2-1}{4}w(\beta)}\Tr_q(\bm{\beta}),
\end{equation}
where $\bm{\beta} = \bm{\beta}_{V_n}$ denotes the representation induced by the braid, and $w(\beta)$ is the writhe of the braid (the number of positive crossings minus the number of negative crossings). In other words, up to a framing factor, the unreduced colored Jones polynomial of a braid closure is given as the state sum determined by the $R$-matrix and the following evaluations and coevaluations: 
\begin{align*}
    \overrightarrow{ev} &: V_n \otimes V_n^* \rightarrow \Bbbk,\quad v_i \otimes v_j^* \mapsto \delta_{ij} q^{\frac{i}{2}},\\
    \overleftarrow{ev} &: V_n^* \otimes V_n \rightarrow \Bbbk,\quad v_i^* \otimes v_j \mapsto \delta_{ij},\\
    \overrightarrow{coev} &: \Bbbk \rightarrow V_n^* \otimes V_n,\quad 1\mapsto \sum_i q^{-\frac{i}{2}}v_i^*\otimes v_i,\\
    \overleftarrow{coev} &: \Bbbk \rightarrow V_n \otimes V_n^*,\quad 1\mapsto \sum_i v_i\otimes v_i^*.
\end{align*}

\begin{figure}[h]
    \centering
    \includegraphics[scale=0.4]{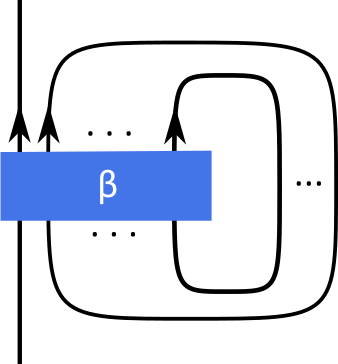}
    \caption{Right-closure of a braid with the leftmost strand open}
    \label{fig:rightclosure}
\end{figure}
For reduced invariants, we should open up a strand as in Figure \ref{fig:rightclosure}. Regarding $\bm{\beta}$ as a matrix, let $\bm{\beta}'$ be the submatrix of $\bm{\beta}$ obtained by restricting to the subspace spanned by vectors of the form $v_{i_1}\otimes \cdots \otimes v_{i_N}$ with fixed $i_1$ and arbitrary $i_2,\cdots,i_N$. Define the \emph{reduced quantum trace} to be
\begin{equation}
    \Tr_q'(\bm{\beta}):=\Tr_q(\bm{\beta}').
\end{equation}
Then the \emph{reduced} $V_n$-colored Jones polynomial of a knot $K$ presented as the closure of a braid $\beta$ is given by
\begin{equation}
    J_K(n;q) = q^{-\frac{n^2-1}{4}w(\beta)}\Tr_q'(\bm{\beta}).
\end{equation}
The choice of $i_1$ doesn't matter, because from the point of view of state sum, the resulting map $V_n\rightarrow V_n$ is central and hence $C\cdot \mathrm{id}$ for some $C \in \Bbbk$. This factor $C$ is the reduced $V_n$-colored Jones polynomial. 

Finally, we should mention that, more generally there is an $R$-matrix acting on $V_n \otimes V_m$: 
\begin{equation}
R(v_i\otimes w_j) = \sum_{k\geq 0} (q^{\frac{1}{2}}-q^{-\frac{1}{2}})^k[k]! \qbin{\frac{n-1-i}{2}}{k}\qbin{\frac{m-1+j}{2}}{k}q^{\frac{1}{4}\qty(ij -k(i-j) -k(k+1))} v_{i+2k}\otimes w_{j-2k}.
\end{equation}
The quantum trace of the corresponding representation gives the colored Jones polynomials where now each strand is colored differently.

\section{Large color $R$-matrix}\label{sec:largecolorR}
\subsection{Highest weight and lowest weight Verma modules $V^h_\infty$ and $V^l_\infty$}
There are two infinite-dimensional modules of $U_q(\mathfrak{sl}_2)$ over $\Bbbk_x := \mathbb{C}(q^{\frac{1}{2}},x^{\frac{1}{2}})$, appearing naturally in the large color limit of $V_n$. We'll formally write $n = \log_q x$, but we'll mostly keep $x$ generic, so $n$ doesn't have to be an integer.

The \emph{highest weight Verma module} $V_{\infty,\lambda}^{h}$ with highest weight $\lambda$ is the following module. 
\[\cdots\overset{e}{\underset{f}{\rightleftharpoons}}V^h_\infty(\lambda-2)\overset{e}{\underset{f}{\rightleftharpoons}} V^h_\infty(\lambda)\]
There is a basis $\{v^j\}_{j\geq 0}$ with $v^j\in V^h_\infty(\lambda - 2j)$ for which the action of $U_q(\mathfrak{sl}_2)$ is given by
\begin{align*}
    e v^j &= [j]v^{j-1},\\
    f v^{j} &= [\lambda-j]v^{j+1},\\
    q^{\frac{h}{2}}v^{j} &= q^{\frac{\lambda - 2j}{2}}v^{j}.
\end{align*}
If we choose $\lambda\in \mathbb{Z} +\mathbb{Z}\log_q x$, then this is a module over $\Bbbk_x$. 

Similarly, the \emph{lowest weight Verma module} $V_{\infty,\lambda}^{l}$ with the lowest weight $\lambda$ is the following module. 
\[V^l_\infty(\lambda)\overset{e}{\underset{f}{\rightleftharpoons}}V^l_\infty(\lambda+2)\overset{e}{\underset{f}{\rightleftharpoons}}\cdots\]
There is a basis $\{v_j\}_{j\geq 0}$ with $v_j\in V^l_\infty(\lambda + 2j)$ for which the action of $U_q(\mathfrak{sl}_2)$ is given by
\begin{align*}
    e v_j &= [-\lambda-j]v_{j+1},\\
    f v_j &= [j]v_{j-1},\\
    q^{\frac{h}{2}}v_j &= q^{\frac{\lambda + 2j}{2}}v_j.
\end{align*}
Again, if we choose $\lambda\in \mathbb{Z} +\mathbb{Z}\log_q x$, this is a module over $\Bbbk_x$. 


\subsection{$R$-matrices on Verma modules}\label{subsec:Rmatrix}
In the large color limit of $V_n$, we naturally get $V_{\infty,\log_q x-1}^{h}$ and $V_{\infty,1-\log_q x}^{l}$. For each of them, there is an $R$-matrix: for $V_{\infty,\log_q x-1}^{h}$, 
\begin{equation}
    \check{R}(v^{i}\otimes v^{j}) = q^{\frac{n^2-1}{4}}x^{-\frac{1}{2}}\sum_{k\geq 0}\qbin{i}{k}\prod_{1\leq l\leq k}\qty(1-x^{-1}q^{j+l})\cdot  x^{-\frac{(i-k)+j}{2}}q^{(i-k)j+\frac{(i-k)k}{2}+\frac{(i-k)+j+1}{2}} v^{j+k}\otimes v^{i-k},
\end{equation}
and for $V_{\infty,1-\log_q x}^{l}$, 
\begin{equation}
    \check{R}(v_j\otimes v_i) = q^{\frac{n^2-1}{4}}x^{-\frac{1}{2}}\sum_{k\geq 0}\qbin{i}{k}\prod_{1\leq l\leq k}\qty(1-x^{-1}q^{j+l}) \cdot x^{-\frac{(i-k)+j}{2}}q^{(i-k)j+\frac{(i-k)k}{2}+\frac{(i-k)+j+1}{2}} v_{i-k}\otimes v_{j+k}.
\end{equation}
Here, the quadratic part, $q^{\frac{n^2-1}{4}}$, is a framing factor. Since we'll only consider $0$-framed knots, we can ignore this factor throughout this paper. 

Observe that the two $R$-matrices are the same up to flip $P$, meaning that the highest weight $R$-matrix for a crossing seen from above the page is the same as the lowest weight $R$-matrix for the crossing seen from below the page. See Figure \ref{fig:highlowduality}.
\begin{figure}[h]
    \centering
    \includegraphics[scale=0.65]{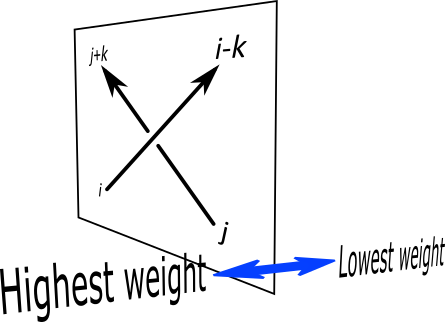}
    \caption{Highest-lowest duality, or Weyl symmetry, of the $R$-matrices}
    \label{fig:highlowduality}
\end{figure}
Also note that only positive powers of $q$ and negative powers of $x$ appear for positive crossings. This is an important property that will be used later in this paper when we study positive braid knots. 

To get their inverses, we flip them left and right, and replace $q$ and $x$ with their inverses. That is, 
\begin{equation}
\check{R}^{-1} = P\check{R}\vert_{\substack{x\rightarrow x^{-1}\\ q\rightarrow q^{-1}}}P.
\end{equation}
Hence only negative powers of $q$ and positive powers of $x$ appear for negative crossings. 

More generally, we use different parameters for the Verma modules assigned to the two strands. The $R$-matrix for $V^h_{\infty,\log_q x-1}\otimes V^h_{\infty,\log_q y-1}$ is
\begin{align}
    \check{R}(v^i\otimes w^j) &= q^{\frac{nm-1}{4}}x^{-\frac{1}{4}}y^{-\frac{1}{4}}\\
    &\quad\times \sum_{k\geq 0}\qbin{i}{k}\prod_{1\leq l\leq k}\qty(1-y^{-1}q^{j+l})x^{-\frac{j}{2}-\frac{k}{4}}y^{-\frac{i-k}{2}+\frac{k}{4}}q^{(i-k)j+\frac{(i-k)k}{2}+\frac{(i-k)+j+1}{2}} w^{j+k}\otimes v^{i-k},\nonumber
\end{align}
and the $R$-matrix for $V^l_{\infty,1-\log_q y}\otimes V^l_{\infty,1-\log_q x}$ is
\begin{align}
    \check{R}(w_j\otimes v_i) &= q^{\frac{nm-1}{4}}x^{-\frac{1}{4}}y^{-\frac{1}{4}}\\
    &\quad\times\sum_{k\geq 0}\qbin{i}{k}\prod_{1\leq l\leq k}\qty(1-y^{-1}q^{j+l})x^{-\frac{j}{2}-\frac{k}{4}}y^{-\frac{i-k}{2}+\frac{k}{4}}q^{(i-k)j+\frac{(i-k)k}{2}+\frac{(i-k)+j+1}{2}} v_{i-k}\otimes w_{j+k}.\nonumber
\end{align}

These large color $R$-matrices satisfy the quantum Yang-Baxter equation 
\begin{equation}
\check{R}_{23}\check{R}_{12}\check{R}_{23} = \check{R}_{12}\check{R}_{23}\check{R}_{12}.
\end{equation}
Therefore they induce an infinite-dimensional representation of the braid group. 
\begin{prop}
    The large color $R$-matrix induces an infinite-dimensional, $U_q(\mathfrak{sl}_2)$-equivariant  representation of the braid group $B_N$ on $V_\infty^{\otimes N}$. 
\end{prop}
\begin{proof}
    We know that for each $x=q^n$ with $n \in \mathbb{Z}_+$, the $R$-matrix satisfies the quantum Yang-Baxter equation. Because each entry of the large color $R$ matrix is a Laurent polynomial in $x$ and $q$, it follows that it satisfies the quantum Yang-Baxter equation as well. 
\end{proof}

\subsection{Semi-classical limit of the $R$-matrix and Burau representation}
In this subsection we briefly review the connection between the $R$-matrix and Burau representation, following \cite{R1}. 
Since the story is similar for the lowest weight $R$-matrix, let's focus on the highest weight $R$-matrix. In the semi-classical limit $q\rightarrow 1$, ignoring the prefactor $q^{\frac{n^2-1}{4}}x^{-\frac{1}{2}}$ for the moment\footnote{The factor $q^{\frac{n^2-1}{4}}$ is a framing factor, and hence will be cancelled out for the $0$-framing. The term $x^{-\frac{1}{2}}$ will be collected into $x^{-\frac{w(\beta)}{2}}$.}, the highest weight $R$-matrix induces the following representation of the braid group on the total weight $1$ subspace (written in basis $v^0\otimes\cdots\otimes v^0\otimes v^1\otimes v^0\otimes\cdots\otimes v^0$).
\[\sigma_i \mapsto \begin{pmatrix}I_{i-1} & 0 & 0 & 0 \\ 0 & x^{-\frac{1}{2}}(x^{\frac{1}{2}}-x^{-\frac{1}{2}}) & x^{-\frac{1}{2}} & 0 \\ 0 & x^{-\frac{1}{2}} & 0 & 0 \\ 0 & 0 & 0 & I_{N-1-i}\end{pmatrix}\]

This representation turns out to be equivalent to the Burau representation. In fact, exactly in this normalization, this representation on the total weight $1$ subspace extends to an algebra morphism on the total space $V_\infty^{\otimes N}$. That is, for $V_\infty^{\otimes N}$, we identify (after the rescaling)
\[v^{i_1}\otimes \cdots \otimes v^{i_N} \leftrightarrow z_1^{i_1}\cdots z_N^{i_N} \in \Bbbk[z_1,\cdots,z_N].\]
Alexander polynomial can be described in terms of this Burau representation. Given a braid $\beta$ with $N$ strands, there corresponds an $N\times N$ matrix $\widetilde{\bm{\beta}}$ coming from the Burau representation. Take the $(N-1)\times(N-1)$ submatrix $\widetilde{\bm{\beta}}'$ corresponding to strands $2, \cdots, N$. Then, the Alexander polynomial is
\[\Delta_K(x) = x^{-\frac{N-1-w(\beta)}{2}}\det(I_{N-1}-\widetilde{\bm{\beta}}'),\]
where $w(\beta)$ is the writhe of the braid. 

Now we're ready to state the connection between the semi-classical limit of the large color Jones polynomial and the Alexander polynomial. Let $\hat{\bm{\beta}}$ be the endomorphism of the algebra $\Bbbk[z_1,\cdots,z_N]$ induced by $\widetilde{\bm{\beta}}$, and let $\hat{\bm{\beta}}'$ be the endomorphism of the subalgebra $\Bbbk[z_2,\cdots,z_N]$ obtained by projection. As we take the quantum trace, we get
\[\lim_{\substack{q\rightarrow 1\\ q^n\text{ fixed}}}J_K(n;q) = x^{\frac{N-1-w(\beta)}{2}}\sum_{w\geq 0}\Tr \hat{\bm{\beta}}'(w) y^w\bigg\vert_{y\rightarrow 1} = \frac{x^{\frac{N-1-w(\beta)}{2}}}{\det(I_{N-1}-y\widetilde{\bm{\beta}}')}\bigg\vert_{y\rightarrow 1} = \frac{1}{\Delta_K(x)}.\]
This was the proof of Melvin-Morton conjecture due to Rozansky. In fact, Rozansky pushed this line of reasoning further to study higher order terms (terms of higher $\hbar$ degree) and proved Melvin-Morton-Rozansky conjecture. See \cite{R1} for the details.

\section{$F_K$ from large color $R$-matrix}\label{sec:main}
In this section, we use the large color $R$-matrix to compute $F_K$ for various knots. Conceptually, it is clear that the large color $R$-matrix would be useful in computing $F_K$, as $F_K$ is, morally, an analytic continuation of the colored Jones polynomials, and the large color $R$-matrix is exactly the analytic continuation of the finite color $R$-matrices. 

\subsection{Closing up the braid}
In the previous sections we have seen that given a braid $\beta$ with $N$ strands, the large color $R$-matrix induces an automorphism $\bm{\beta}=\bm{\beta}_{V_\infty}$ of $V_\infty^{\otimes N}$. As we close up the braid, we need to take the quantum trace of this map. This is an infinite sum, and we would like to make sense of it. 

\subsubsection{An invariant of transverse knots from the large color $R$-matrix}
Let's suppose that we are using the the highest weight Verma modules (i.e.\ $\bm{\beta}=\bm{\beta}_{V_\infty^h}$) and that the quantum trace $\Tr_q \bm{\beta}$ is absolutely convergent, so that it converges to a well-defined formal power series regardless of the order of summation in the state sum. 
Then, $\Tr_q \bm{\beta}$ is invariant under conjugation and braid isotopy. Moreover, it is invariant under positive stabilization (Figure \ref{fig:positiveMarkov}), as positive stabilization only involves finite number of reorderings at each degree.\footnote{If we were using the lowest weight Verma modules, then it would be invariant under negative stabilization instead.} 
\begin{figure}[h]
    \centering
    \includegraphics[scale=0.5]{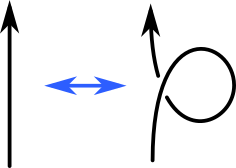}
    \caption{Positive stabilization}
    \label{fig:positiveMarkov}
\end{figure}
Therefore, the transverse Markov theorem \cite{OS, W} immediately implies the following. 
\begin{prop}\label{prop:transverse}
    Whenever $\Tr_q \bm{\beta}$ is absolutely convergent, it is an invariant of the corresponding transverse knot.\footnote{For an introduction to transverse knots, see \cite{Etnyre}.} 
\end{prop}

\subsubsection{Stratification of the trace}
As we will see, the quantum trace is absolutely convergent for positive braids, and this will be related to $F_K$. Unfortunately, in general, the quantum trace is not always absolutely convergent. 
How do we make sense of the quantum trace or the state sum when it is not absolutely convergent? We need to specify the order of summation, and one natural choice comes from the stratification of the trace. 

As noted in \cite{R1}, the quantum trace is naturally stratified by the total weight. Let's use the highest weight Verma modules for the moment. We define the total weight of the vector $v^{i_1}\otimes \cdots \otimes v^{i_N}$ to be $w = \sum_{1\leq k\leq N} i_k$, so that the total weight $0$-vector is the highest weight vector (vacuum). The weight $w$-subspace of $V_\infty^{\otimes N}$ is the subspace spanned by vectors of total weight $w$. Note that each weight $w$-subspace of $V_\infty^{\otimes N}$ is invariant under $\bm{\beta}$. Let's call this subrepresentation $\bm{\beta}(w)$. 
Then the stratified trace is given by
\begin{equation}
    \Tr_{q,\eta}\bm{\beta} := x^{\frac{N}{2}}q^{-\frac{N}{2}}\sum_{w\geq 0}\Tr \bm{\beta}(w) q^{-w} \eta^w.
\end{equation}
In the limit $\eta\rightarrow 1$ this is the usual quantum trace, with a specified order of summation. As a formal power series in $\eta$, each coefficient, $\Tr \bm{\beta}(w)$, is a Laurent polynomial in $q$ and $x$. The fact that each subspace $V_\infty^{\otimes N}(w)$ is finite dimensional immediately implies the following. 
\begin{prop}
    $\Tr_{q,\eta}\bm{\beta}$ is invariant under conjugation of the braid. That is, $\Tr_{q,\eta}\bm{\beta}$ is an invariant of the annular braid closure in the solid torus. 
\end{prop}
If we naively set $\eta=1$, this is not a well-defined power series in $q$ and $x$ in general, but still it makes sense in some specializations: 
\begin{itemize}
    \item For $x=q^n$, $n\geq 1$, with generic $q$, the sum becomes $\Tilde{J}_K(n;q)$, the $n$-colored Jones polynomial.
    \item For $x=q^n$ with $n\leq 0$ and $q$ a root of unity, the sum is finite and is well-defined. Let's denote this invariant $\Tilde{J}_K(n;q)$. In particular, for $x=1$, this is the Kashaev's invariant. Moreover, there's the Weyl symmetry
    \begin{equation}
    \Tilde{J}_K(n;q) = \Tilde{J}_K(-n;q)
    \end{equation}
    for every root of unity. 
\end{itemize}
This $x=1$ specialization is the one that we denoted by $\langle K \rangle$ in Corollary \ref{strangeidentity}. Note that at each root of unity, not only the values of $\langle K \rangle$ but also derivatives of arbitrary order makes sense.

Before getting into the actual computation of the quantum trace, let's fix our notational convention first. 
Unless otherwise stated, we'll use the right-closure as in Figure \ref{fig:rightclosure}.\footnote{Note that, due to the highest-lowest duality mentioned in Section \ref{subsec:Rmatrix}, using the right-closure with highest (lowest, resp.) weight Verma modules is the same as using the left-closure with lowest (higest, resp.) weight Verma modules.}
Because we do not close the leftmost strand, we need to use the reduced version of the (stratified) quantum trace, given by
\begin{align}
        \Tr_{q,\eta}'\bm{\beta}_{V_\infty^h} &:= x^{\frac{N-1}{2}}q^{-\frac{N-1}{2}}\sum_{w\geq 0}\Tr \bm{\beta}_{V_\infty^h}'(w) q^{-w} \eta^w,\\
        \Tr_{q,\eta}'\bm{\beta}_{V_\infty^l} &:= x^{-\frac{N-1}{2}}q^{\frac{N-1}{2}}\sum_{w\geq 0}\Tr \bm{\beta}_{V_\infty^l}'(w) q^{w} \eta^w,
\end{align}
and
\begin{equation}\label{eq:limitofstrat}
    \Tr_q'\bm{\beta} = \lim_{\eta \rightarrow 1}\Tr_{q,\eta}'\bm{\beta},
\end{equation}
where $\bm{\beta}'$ denotes the induced map on $V_\infty^{\otimes N-1}$ corresponding to the right-closure (in the same manner we did for finite-dimensional representations in Section \ref{subsec:finRmatrix}).
In our convention for braid presentation, we write from bottom to top; for instance, $\sigma_1\sigma_2$ means $\sigma_1$ on the bottom and $\sigma_2$ on the top. Finally, $F^+_K$, $F^-_K$, $F^{\pm}_K$ denote the positive expansion (power series expansion in $x$), negative expansion (power series expansion in $x^{-1}$), and the balanced expansion of $F_K$ (average of the former two), respectively. Weyl symmetry implies that $F_K^+(x,q) = -F_K^-(x^{-1},q)$. The balanced expansion $F_K^{\pm} = \frac{1}{2}\qty(F_K^+ + F_K^-)$ is the one used in \cite{GM}. 

\begin{rmk}
It is important to keep in mind that the resulting map $V_\infty^h\rightarrow V_\infty^h$ for the 1-1 tangle is \emph{not} central in general. The simplest example where this is not central is the negative stabilization. Nevertheless, it turns out it is central for positive braid knots as we are about to see and various other examples. 
\end{rmk}

\subsection{Positive braid knots}
It turns out that the infinite sum converges absolutely for \emph{positive braid knots} $K$, and each coefficient of $F_K$ is a finite polynomial.\footnote{Note that this is consistent with the fact that positive braid knots are fibered and thus have monic Alexander polynomials, so the power series expansion of $\frac{1}{\Delta_K(x)}$ has integer coefficients.} Below, we prove Theorem \ref{pbthm} mentioned in the introduction. 
\begin{proof}[Proof of Theorem \ref{pbthm}]
    Let $K$ be a positive braid knot represented by a positive braid $\beta$. We will prove that 
    \begin{equation}\label{eq:pbthm}
        (x^{\frac{1}{2}}-x^{-\frac{1}{2}})\Tr_q' \bm{\beta}_{V_\infty^h} = F_K^-(x,q),
    \end{equation}
    and that if we write $F_K^-(x,q) = -x^{-\frac{1}{2}}\sum_{m\geq 0}f_m^K(q)x^{-m}$, each coefficient $f_m^K(q)$ is a polynomial in $q$. 
    
    We first show that the quantum trace $\Tr_q' \bm{\beta}_{V_\infty^h}$ converges absolutely, and that $\Tr_q' \bm{\beta}_{V_\infty^h} \in \mathbb{Z}[q][[x^{-1}]]$, thanks to positivity of the braid. Recall that\footnote{Here we have erased the quadratic part, as it will be cancelled out in the end anyway since we only consider $0$-framed knots.} 
    \begin{equation}
    \check{R}(v^{i}\otimes v^{j}) = \sum_{k\geq 0}\qbin{i}{k}\prod_{1\leq l\leq k}\qty(1-x^{-1}q^{j+l})\cdot  x^{-\frac{(i-k)+j}{2}}q^{(i-k)j+\frac{(i-k)k}{2}+\frac{(i-k)+j+1}{2}} v^{j+k}\otimes v^{i-k}.
    \end{equation}
    Observe that the highest $x$-degree of the coefficient is $-\frac{(i-k)+j}{2}$ and the lowest $q$-degree is $((i-k)+\frac{1}{2})(j+\frac{1}{2})+\frac{1}{4}$. In other words, for each crossing, the bigger the weight $i-k$ of the top right strand and the bigger the weight $j$ of bottom right strand, the smaller the $x$-degree gets and the bigger the $q$-degree gets.\footnote{Since these bounds come from only the right strands, it is crucial that we use the right-closure and leave the left-most strand open (Figure \ref{fig:rightclosure}).}
    
    \begin{figure}
    \centering
    \begin{minipage}{.5\textwidth}
        \centering
        \includegraphics[scale=0.5]{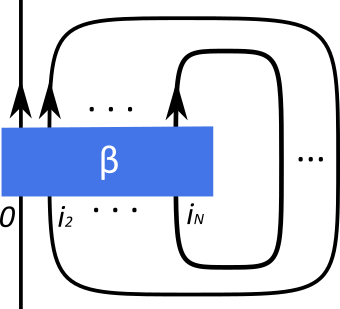}
        \captionof{figure}{}
        \label{fig:prooffigure1}
    \end{minipage}%
    \begin{minipage}{.5\textwidth}
        \centering
        \includegraphics[scale=0.5]{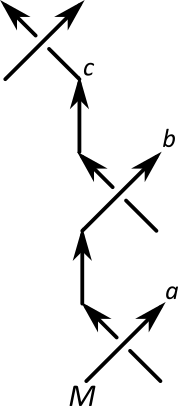}
        \captionof{figure}{}
        \label{fig:prooffigure2}
    \end{minipage}
    \end{figure}
    Consider a state where the outer strands are assigned $v^0, v^{i_2}, \cdots, v^{i_N}$ from left to right; see Figure \ref{fig:prooffigure1}. 
    Suppose that $M = i_k = \max\{i_2,\cdots,i_N\}$. The braid $\beta$, thought of as a word in $\sigma_1,\cdots,\sigma_{N-1}$, must have a $\sigma_{k-1}$ because the closure of $\beta$ is a knot. Let's focus on appearances of $\sigma_k$ before the first $\sigma_{k-1}$. The following argument works regardless of the number of $\sigma_k$'s, but for the purpose of illustration, let's say there are two $\sigma_{k}$'s before the first $\sigma_{k-1}$ appears; see Figure \ref{fig:prooffigure2}. 
    In the figure, $a$, $b$ and $c$ denotes that the corresponding strands are assigned $v^a$, $v^b$ and $v^c$, respectively. It is straightforward to see that $a+b+c \geq M$. Our previous observation on the bound on the $x$- and $q$-degrees of the $R$-matrix immediately implies that the weight of such a state has $x$-degree at most $-\frac{M}{2}$ and $q$-degree at least $\frac{M}{2}$. Because for each fixed $M$ there are only finitely many possible such states, it immediately follows that the quantum trace converges absoluately, and $\Tr_q' \bm{\beta}_{V_\infty^h} \in \mathbb{Z}[q][[x^{-1}]]$. 
    
    It is known \cite{E} that any two positive braid presentations of a given positive braid knot are related via a sequence of positive stabilizations. Since the quantum trace is invariant under positive stabilization (Proposition \ref{prop:transverse}), $\Tr_q' \bm{\beta}_{V_\infty^h}$ is a well-defined invariant of positive braid knots. 
    
    Now it suffices to prove that the perturbative series of this invariant agrees with that of the large color expansion (Melvin-Morton-Rozansky expansion) of colored Jones polynomials, which basically follows from the result of Rozansky \cite{R1}. 
\end{proof}
We should mention that, in a similar manner, if $K$ is a negative braid knot, we have
\begin{equation}
    (x^{\frac{1}{2}}-x^{-\frac{1}{2}})\Tr_q' \bm{\beta}_{V_\infty^l} = F_K^+(x,q) = F_{m(K)}^{+}(x,q^{-1}),
\end{equation}
where $m(K)$ denotes the mirror knot. 

\begin{rmk}
Our proof also shows that we could have chosen any other vector on the open strand and still get a quantum trace that lives in $\mathbb{Z}[q][[x^{-1}]]$ whose perturbative expansion is the large color expansion of the colored Jones. Since a polynomial in $q$ whose $\hbar$-expansion vanishes should be identically $0$, it follows that the map $V_\infty^h\rightarrow V_\infty^h$ we get from the 1-1 tangle as in Figure \ref{fig:prooffigure1} is central. 
\end{rmk}

\subsubsection{Positive braid knots up to 10 crossings}
The simplest examples of positive braid knots are positive torus knots. In fact, they are the simplest type of knots, and their $F_K$'s are given in \cite{GM}. Even their higher rank analogues are known (up to overall $q$-power); see \cite{P}. So here let's focus on non-torus knots. Up to 10 crossings, there are two non-torus positive braid knots: $\mathbf{10}_{139}$ and $\mathbf{10}_{152}$.

$\mathbf{10}_{139}$ has a positive braid presentation $\beta = \sigma_1^4\sigma_2\sigma_1^3\sigma_2^2$. 
Computing the quantum trace, we get
\begin{align*}
    F_{\mathbf{10}_{139}}^{-}(x,q) &= q^4 x^{-\frac{7}{2}} -2q^6 x^{-\frac{13}{2}} +q^7 x^{-\frac{15}{2}} -q^8 x^{-\frac{17}{2}} +(2q^9+q^{10}) x^{-\frac{19}{2}} +O(x^{-\frac{21}{2}}).
\end{align*}

$\mathbf{10}_{152}$ has a positive braid presentation $\beta = \sigma_1^3\sigma_2^2\sigma_1^2\sigma_2^3$. 
Computing the quantum trace, we get
\begin{align*}
    F_{\mathbf{10}_{152}}^{-}(x,q) &= q^4 x^{-\frac{7}{2}} +q^5 x^{-\frac{11}{2}} -3q^6 x^{-\frac{13}{2}} +(q^6+2q^7)x^{-\frac{15}{2}} +O(x^{-\frac{17}{2}}).
\end{align*}

\subsection{Fibered strongly quasi-positive braid knots}
Fibered strongly quasi-positive braid knots are more general than just positive braid knots, but they are still special in a sense that there is a unique transverse structure with the maximal transverse self-linking number; see \cite{E}. Based on examples we list below, we conjecture that Theorem \ref{pbthm} extends to fibered strongly quasi-positive braid knots. 
\begin{conj}
    Let $K$ be a fibered strongly quasi-positive braid knot. Then 
    \begin{equation}
        (x^{\frac{1}{2}}-x^{-\frac{1}{2}})\Tr_q' \bm{\beta}_{V_\infty^h} = F_K^-(x,q).
    \end{equation}
    Moreover, if we write $F_K^-(x,q) = -x^{-\frac{1}{2}}\sum_{m\geq 0}f_m^K(q)x^{-m}$, each coefficient $f_m^K(q)$ is a Laurent polynomial in $q$. 
\end{conj}

Up to 10 crossings, there are 3 fibered strongly quasi-positive braid knots which are not positive braids: $m(\mathbf{10}_{145})$, $\mathbf{10}_{154}$ and $\mathbf{10}_{161}$. See Table \ref{tab:SQpositive}. 
\begin{table}[h]
    \begin{tabular}{||c c||}
    \hline
    Knot & SQ-positive braid \\ [0.5ex] 
    \hline\hline
    $\mathbf{3}_1$ & $\sigma_1^3$ \\
    $\mathbf{5}_1$ & $\sigma_1^5$ \\
    $\mathbf{7}_1$ & $\sigma_1^7$ \\
    $\mathbf{8}_{19}$ & $\sigma_1^3\sigma_2\sigma_1^3\sigma_2$ \\
    $\mathbf{9}_1$ & $\sigma_1^9$ \\
    $\mathbf{10}_{124}$ & $\sigma_1^5\sigma_2\sigma_1^3\sigma_2$ \\
    $\mathbf{10}_{139}$ & $\sigma_1^4\sigma_2\sigma_1^3\sigma_2^2$ \\
    $m(\mathbf{10}_{145})$ & $\sigma_3\sigma_2^2\sigma_1\sigma_2^{-1}\sigma_3^2\sigma_2^2\sigma_1\sigma_2^{-1}$ \\
    $\mathbf{10}_{152}$ & $\sigma_1^3\sigma_2\sigma_1^2\sigma_2^3$ \\
    $\mathbf{10}_{154}$ & $\sigma_1^2\sigma_2\sigma_1^{-1}\sigma_2\sigma_1\sigma_3\sigma_2^3\sigma_3$ \\
    $\mathbf{10}_{161}$ & $\sigma_1^3\sigma_2\sigma_1^{-1}\sigma_2\sigma_1^2\sigma_2^2$ \\
    \hline
    \end{tabular}
    \caption{\label{tab:SQpositive}All fibered strongly quasi-positive braid knots up to 10 crossings \cite{KnotInfo}\protect\footnotemark}
\end{table}
\footnotetext{When it comes to Alexander-Briggs notation of knots, we follow the convention used in \cite{KnotInfo}, which is sometimes the mirror of the one in \cite{KnotAtlas}.}

$m(\mathbf{10}_{145})$ has a strongly quqasi-positive braid presentation $\beta = \sigma_3\sigma_2^2\sigma_1\sigma_2^{-1}\sigma_3^2\sigma_2^2\sigma_1\sigma_2^{-1}$. 
Computing the quantum trace, we get
\begin{align*}
    F_{m(\mathbf{10}_{145})}^{-}(x,q) &= q^2 x^{-\frac{3}{2}} -2q^2 x^{-\frac{5}{2}} +(2q+2q^3+q^4)x^{-\frac{7}{2}}\\
    &\quad +(-2q^{-1}-2q^2-2q^3-2q^4-4q^5)x^{-\frac{9}{2}} +O(x^{-\frac{11}{2}}).
\end{align*}

$\mathbf{10}_{154}$ has a strongly quasi-positive braid presentation $\beta = \sigma_1^2\sigma_2\sigma_1^{-1}\sigma_2\sigma_1\sigma_3\sigma_2^3\sigma_3$. Computing the quantum trace, we get
\begin{align*}
    F_{\mathbf{10}_{154}}^{-}(x,q) &= q^3 x^{-\frac{5}{2}} -q^3 x^{-\frac{7}{2}} +(q^2+3q^4)x^{-\frac{9}{2}} +(-1-2q^3-2q^4-5q^5-q^6)x^{-\frac{11}{2}} +O(x^{-\frac{13}{2}}).
\end{align*}

$\mathbf{10}_{161}$ ``\emph{the Perko pair}'' has a strongly quqasi-positive braid presentation $\beta = \sigma_1^3\sigma_2\sigma_1^{-1}\sigma_2\sigma_1^2\sigma_2^2$. Computing the quantum trace, we get
\begin{align*}
    F_{\mathbf{10}_{161}}^{-}(x,q) &= q^3 x^{-\frac{5}{2}} - q^3 x^{-\frac{7}{2}} +(q^2+q^4)x^{-\frac{9}{2}} +(-1-q^3-q^4-2q^5)x^{-\frac{11}{2}} +O(x^{-\frac{13}{2}}).
\end{align*}

\subsection{Some positive double twist knots}
So far we have only dealt with fibered knots. For any non-fiberd knot $K$ with non-monic Alexander polynomial, the power series expansion of $\frac{1}{\Delta_K(x)}$ has non-integer coefficients, meaning that $F_K$, as a power series in $x$, should have coefficients which are not any more Laurent polynomials but power series in $q$. 
It turns out that the method or large color $R$-matrix is quite robust that it works for many non-fibered knots as well. In this subsection, we experimentally compute $F_K$ for some double twist knots. One subtle point that we should emphasize is that we cannot guarantee absolute convergence of the infinite sum any more, so all the quantum trace in this subsection should be understood as the limit of stratified quantum trace in the sense of \eqref{eq:limitofstrat}.

Let $K_{m,n}$ be the double twist knot with $m$ and $n$ full-twists; see Figure \ref{fig:dt}. We will call a double twist knot \emph{positive} if it is of the form $K_{m,n}$ or $K_{m+\frac{1}{2},-n}$ for some positive integers $m, n$; see Table \ref{tab:doubletwist}. 
\begin{figure}[h]
    \centering
    \includegraphics[scale=0.6]{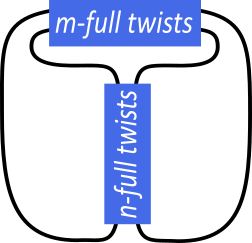}
    \caption{The double twist knot $K_{m,n}$}
    \label{fig:dt}
\end{figure}
\begin{table}[h]
    \begin{tabular}{||c c||}
    \hline
    Knot & double-twist notation \\ [0.5ex] 
    \hline\hline
    $m(\mathbf{3}_1)$ & $K_{1,1}$ \\
    $\mathbf{4}_1$ & $K_{1,-1}$ \\
    $m(\mathbf{5}_2)$ & $K_{2,1}$ \\
    $m(\mathbf{6}_1)$ & $K_{2,-1}$ \\
    $m(\mathbf{7}_2)$ & $K_{3,1}$ \\
    $m(\mathbf{7}_3)$ & $K_{\frac{3}{2},-2}$ \\
    $m(\mathbf{7}_4)$ & $K_{2,2}$ \\
    $m(\mathbf{8}_1)$ & $K_{3,-1}$ \\
    $m(\mathbf{9}_2)$ & $K_{4,1}$ \\
    $m(\mathbf{9}_3)$ & $K_{\frac{3}{2},-3}$ \\
    $m(\mathbf{9}_4)$ & $K_{\frac{5}{2},-2}$ \\
    $m(\mathbf{9}_5)$ & $K_{3,2}$ \\
    \hline
    \end{tabular}
    \caption{\label{tab:doubletwist} Some positive double twist knots with small number of crossings}
\end{table}
The two smallest double twist knots, trefoil ($K_{1,1}$) and figure-eight ($K_{1,-1}$), were already studied in \cite{GM}, so let's move on to the next simplest example, $K_{2,1} = m(\mathbf{5}_2)$. 
Using the braid presentation $\beta = \sigma_2^{-3}\sigma_1^{-1}\sigma_2\sigma_1^{-1}$, we get
\begin{equation}
    (x^{\frac{1}{2}}-x^{-\frac{1}{2}})\Tr_q'\bm{\beta_{V_\infty^l}} = F_{m(\mathbf{5}_2)}^{+}(x,q) = x^{\frac{1}{2}}\sum_{j\geq 0} f_j^{m(\mathbf{5}_2)}(q) x^j,
\end{equation}
where the first few terms are given by
\begin{align*}
    f_0^{m(\mathbf{5}_2)}(q) &= -q^{-1} +1 -q^2 +q^5 -q^9 +q^{14} -q^{20} +q^{27} -O(q^{35}),\\
    f_1^{m(\mathbf{5}_2)}(q) &= -q^{-1} +1 +q -q^2 -q^3 -q^4 +q^5 +q^6 +q^7 +q^8 -O(q^9),\\
    f_2^{m(\mathbf{5}_2)}(q) &= -q^{-1} +2 +q -q^2 -2q^3 -2q^4 +q^5 +q^6 +3q^7 +2q^8 -q^{10} -O(q^{11}),\\
    f_3^{m(\mathbf{5}_2)}(q) &= 2 +q -2q^2 -2q^3 -3q^4 +2q^6 +4q^7 +4q^8 +2q^9 -3q^{11} -O(q^{12}).
\end{align*}
From the pattern, we find the closed expressions for the first two coefficients:
\begin{align*}
    f_0^{m(\mathbf{5}_2)}(q) &= -q^{-1}\sum_{j\geq 0}(-1)^j q^{\frac{j(j+1)}{2}},\\ 
    f_1^{m(\mathbf{5}_2)}(q) &= -q^{-1}\sum_{j\geq 0}(-1)^j q^{\frac{j(j+1)}{2}}\frac{1-q^{j+1}}{1-q}.\\
\end{align*}
Given these closed expressions, it is easy to check that their $\hbar$-expansions indeed agree with the ones computed in \cite{CGPS}. Moreover, one can numerically check that it is annihilated by the (unreduced) quantum $A$-polynomial given by \footnote{The data for the quantum $A$-polynomial can be found in \cite{GS, GK, NRZS}}
\begin{equation}
    \hat{A}_{m(\mathbf{5}_2)}(\hat{x},\hat{y},q) = a_0(x,q) + a_1(x,q)\hat{y} + a_2(x,q)\hat{y}^2 + a_3(x,q)\hat{y}^3 + a_4(x,q)\hat{y}^4,
\end{equation}
where
\begin{align*}
    a_0(x,q) &= -q^9 x^7 \left(q^3 x+1\right) \left(q^5 x^2-1\right) \left(q^7 x^2-1\right),\\
    a_1(x,q) &= q^{11/2} x^2 (q x+1) \left(q^3 x+1\right) \left(q^7 x-1\right) \left(q^9 x^6+q^8 x^6-q^8 x^4-3 q^7 x^5-q^7 x^4-q^6 x^5\right.\\
    &\quad +2 q^6 x^4+2 q^6 x^3+q^5 x^4-q^5 x^2-q^4 x^3-q^3 x^4-q^3 x^3+q^3 x^2+2 q^2 x^3+2 q^2 x^2\\
    &\quad\left.-q x^2-2 q x+1\right),\\
    a_2(x,q) &= -q^2 \left(q^2 x-1\right) \left(q^2 x+1\right) \left(q x^2-1\right) \left(q^7 x^2-1\right) \left(q^{12} x^6+q^{11} x^5-2 q^{10} x^5-3 q^9 x^4\right.\\
    &\quad +2 q^8 x^4-q^8 x^3+2 q^7 x^3-q^6 x^4-4 q^6 x^3-q^5 x^3-3 q^5 x^2+q^4 x^3+2 q^4 x^2+q^3 x\\
    &\quad\left.-q^2 x^2-2 q^2 x+1\right),\\
    a_3(x,q) &= q^{1/2} \left(q x^2-1\right) (q x+1) \left(q^3 x+1\right) \left(q^{16} x^6-2 q^{13} x^5-q^{13} x^4+q^{11} x^4+2 q^{10} x^4\right.\\
    &\quad +2 q^{10} x^3-q^9 x^4-q^8 x^3-q^8 x^2-q^7 x^3-q^7 x^2+2 q^6 x^3+2 q^6 x^2+q^5 x^2-q^3 x^2\\
    &\quad\left.-3 q^3 x-q^2 x+q+1\right),\\
    a_4(x,q) &= (q x+1) \left(q x^2-1\right) \left(q^3 x^2-1\right).\\
\end{align*}
In fact, annihilation of $F_K^+$ by $\hat{A}_K$ implies that all the coefficients $f_m^K(q)$ are determined by the first few of them. It the case of $m(\mathbf{5}_2)$, it turns out that all the coefficients are $\mathbb{Q}(q)$-linear combinations of the first two.\footnote{It seems that for the twist knot $K_{m,1}$, the first $|m|$ coefficients determine all the other ones using recursion.} For instance, 
\begin{align*}
    f_2^{m(\mathbf{5}_2)}(q) &= \frac{1+q-q^2}{-q+q^3}f_0^{m(\mathbf{5}_2)}(q) + \frac{-1-2q+q^2}{-q+q^3}f_1^{m(\mathbf{5}_2)}(q),\\
    f_3^{m(\mathbf{5}_2)}(q) &= \frac{-2-q-q^2+2q^3+q^4-q^5}{q^2-q^4-q^5+q^7}f_0^{m(\mathbf{5}_2)}(q) + \frac{2+q+3q^2-q^3-q^6}{q^2-q^4-q^5+q^7}f_1^{m(\mathbf{5}_2)}(q),
\end{align*}
and using recursion it is easy to find closed expressions of the coefficients up to arbitrary order. 

We give another example, $K_{\frac{3}{2},-2} = m(\mathbf{7}_3)$. Using the braid presentation $\beta = \sigma_2^{-5}\sigma_1^{-1}\sigma_2\sigma_1^{-1}$, we get
\begin{equation}
    (x^{\frac{1}{2}}-x^{-\frac{1}{2}})\Tr_q'\bm{\beta_{V_\infty^l}} = F_{m(\mathbf{7}_3)}^{+}(x,q) = x^{\frac{1}{2}}\sum_{j\geq 0} f_j^{m(\mathbf{7}_3)}(q) x^j
\end{equation}
where
\begin{align*}
    f_0^{m(\mathbf{7}_3)} &= 0,\\
    f_1^{m(\mathbf{7}_3)} &= -q^{-2}+q^{-1}-q+q^4-q^8+q^{13}-q^{19}+q^{26}-q^{34}+q^{43}-O(q^{53}),\\
    f_2^{m(\mathbf{7}_3)} &= -q^{-2}+q^{-1}+1-q-q^2-q^3+q^4+q^5+q^6+q^7-q^8-q^9-q^{10}-q^{11}-O(q^{12}),\\
    f_3^{m(\mathbf{7}_3)} &= q^{-3}-2q^{-2}+q^{-1}+2-q^2-3q^3+2q^6+3q^7+q^8-q^{10}-2q^{11}-4q^{12}-q^{13}-O(q^{14}),\\
    f_4^{m(\mathbf{7}_3)} &= q^{-3}-2q^{-2}+2+q+q^2-3q^3-2q^4-2q^5+3q^7+3q^8+3q^9+2q^{10}-4q^{12}-O(q^{13}),\\
    f_5^{m(\mathbf{7}_3)} &= -q^{-5}+q^{-4}+2q^{-3}-3q^{-2}-q^{-1}+2q+3q^2-5q^5-3q^6-q^7+q^8+2q^9+O(q^{10}),
\end{align*}
and so on.

\subsubsection{General expressions from Lovejoy-Osburn}
We note that a general result for positive double twist knots follows from the work \cite{LO1, LO2} of Lovejoy and Osburn. Slight modification of Theorem 1.1 in \cite{LO1} gives us the following formula of $F_K$ for double twist knots $K_{m,p}$ with $m,p>0$ full twists: 
\begin{align}
    \frac{F_{K_{m,p}}^+(x,q)}{x^{\frac{1}{2}}-x^{-\frac{1}{2}}} &= q^{-1}x\sum_{0\leq n_1\leq \cdots \leq n_{2mp-1}}\qty(q^{-1}x)_{n_{2mp-1}}(-1)^{n_{2mp-1}}q^{\binom{n_{2mp-1}+1}{2}}\prod_{\substack{1\leq i<j\leq 2mp-1 \\ m\nmid i}}q^{-\epsilon_{i,j,m}n_in_j}\\
    &\quad\quad \times \prod_{i=1}^{2p-1}(-1)^{n_{mi}}x^{(-1)^{i+1}n_{mi}}q^{-\binom{n_{mi}+1}{2}}\prod_{i=1}^{2mp-2}q^{n_in_{i+1}-\gamma_{i,m}n_i}\qbin{n_{i+1}}{n_i}, \nonumber
\end{align}
where 
\begin{equation}
    \epsilon_{i,j,m} := \begin{cases}1 &\text{if }j\equiv -i\text{ or }-i-1\mod{2m}\\ -1 &\text{if }j\equiv i\text{ or }i-1\mod{2m}\\0&\text{otherwise}\end{cases},
\end{equation}
and
\begin{equation}
    \gamma_{i,m} := \begin{cases} 1 &\text{if }i\equiv 1,\cdots,m-1\mod{2m} \\ -1 &\text{otherwise}\end{cases}.
\end{equation}

Similarly, from a slight modification of Theorem 1.2 in \cite{LO2}, we get the following formula of $F_K$ for $K_{m+\frac{1}{2},-p}$ with positive integers $m,p$: 
\begin{align}
    \frac{F^+_{K_{m+\frac{1}{2},-p}}(x,q)}{x^{\frac{1}{2}}-x^{-\frac{1}{2}}} &= q^{-p}x^p \sum_{0\leq n_1\leq \cdots\leq n_{(2m+1)p}}\qty(q^{-1}x)_{n_{(2m+1)p}}(-1)^{n_{(2m+1)p}}q^{\binom{n_{(2m+1)p}+1}{2}}\\
    &\quad\times \prod_{\substack{1\leq i<j\leq (2m+1)p\\ (2m+1)\nmid i\\ j\not\equiv m+1(\text{mod }2m+1)}}q^{-\Delta_{i,j,k}n_in_j}\prod_{\substack{i=1\\i\equiv m+1,2m+1(\text{mod }2m+1)}}^{(2m+1)p-1}(-1)^{n_i}x^{n_i}q^{-\binom{n_i+1}{2}}\nonumber\\
    &\quad\times \prod_{i=1}^{(2m+1)p-1}q^{-\beta_{i,m}n_i}\qbin{n_{i+1}}{n_i},\nonumber
\end{align}
where 
\begin{equation}
    \Delta_{i,j,m} := \begin{cases}1 &\text{if }j\equiv -i\text{ or }-i+1\mod{2m+1}\\ -1 &\text{if }j\equiv i\text{ or }i+1\mod{2m+1}\\0&\text{otherwise}\end{cases},
\end{equation}
and
\begin{equation}
    \beta_{i,m} := \begin{cases} 1 &\text{if }i\equiv 1,\cdots,m\mod{2m+1}\\ -1 &\text{if }i\equiv m+1,\cdots,2m\mod{2m+1} \\ 0 &\text{if }i\equiv 0\mod{2m+1}\end{cases}.
\end{equation}

\subsection{Other examples}
Here we just give two more examples, $m(\mathbf{7}_5)$ and $m(\mathbf{8}_{15})$, which are neither fibered nor double twist. 

Using the braid presentation $\beta = \sigma_2^{-4}\sigma_1^{-1}\sigma_2\sigma_1^{-2}$ of $m(\mathbf{7}_5)$, we get
\begin{equation}
    (x^{\frac{1}{2}}-x^{-\frac{1}{2}})\Tr_q'\bm{\beta}_{V_\infty^l} = F_{m(\mathbf{7}_5)}^{+}(x,q) = x^{\frac{1}{2}}\sum_{j\geq 0} f_j^{m(\mathbf{7}_5)}(q) x^j,
\end{equation}
where the first few terms are
\begin{align*}
    f_0^{m(\mathbf{7}_5)} &= 0,\\
    f_1^{m(\mathbf{7}_5)} &= -q^{-2}+q^{-1}-q+q^4-q^8+q^{13}-q^{19}+q^{26}-q^{34}+O(q^{43}),\\
    f_2^{m(\mathbf{7}_5)} &= -2q^{-2}+2q^{-1}+2-2q-2q^2-2q^3+2q^4+2q^5+2q^6+2q^7-O(q^8),\\
    f_3^{m(\mathbf{7}_5)} &= q^{-3}-4q^{-2}+3q^{-1}+5-4q^2-8q^3-q^4+q^5+O(q^6).
\end{align*}

Using the braid presentation $\beta = \sigma_1^{-1}\sigma_3^{-1}\sigma_2^{-3}\sigma_1^{-2}\sigma_2\sigma_3\sigma_2^{-1}\sigma_3^{-1}$ of $m(\mathbf{8}_{15})$, we get
\begin{equation}
    (x^{\frac{1}{2}}-x^{-\frac{1}{2}})\Tr_q'\bm{\beta}_{V_\infty^l} = F_{m(\mathbf{8}_{15})}^{+}(x,q) = x^{\frac{1}{2}}\sum_{j\geq 0} f_j^{m(\mathbf{8}_{15})}(q) x^j,
\end{equation}
and the first few terms are
\begin{align*}
    f_0^{m(\mathbf{8}_{15})} &= 0,\\
    f_1^{m(\mathbf{8}_{15})} &= -q^{-2}+2q^{-1}-1-3q+2q^2+2q^3+3q^4-3q^5-4q^6-q^7-O(q^8),\\
    f_2^{m(\mathbf{8}_{15})} &= -3q^{-2}+6q^{-1}+1-12q-2q^2+7q^3+19q^4+4q^5-14q^6-O(q^7),\\
    f_3^{m(\mathbf{8}_{15})} &= -6q^{-2}+15q^{-1}+5-27q-21q^2+5q^3+59q^4+O(q^5).
\end{align*}

\section{Some surgeries}\label{sec:surgeries}
Same $3$-manifold can appear in many different guises. In this section we test our Conjecture \ref{conj:linksurg} through some consistency checks, using different surgery presentations of the same manifold. Moreover we use the surgery formula to reverse-engineer $F_K$ for simple links. 

\subsection{Some consistency checks}
\subsubsection{Exceptional surgeries on $m(\mathbf{5}_2)$}
We start by studying our quintessential non-fibered knot $m(\mathbf{5}_2)$. 
The $0, -1, -2, -3$-surgeries on $m(\mathbf{5}_2)$ are plumbed $3$-manifolds. The case of $0$-surgery was already studied in \cite{CGPS}, so let's focus on $-1, -2, -3$-surgeries. 

The $-1$-surgery on $m(\mathbf{5}_2)$ can be seen in different guises; see Figure \ref{fig:Sigma2311}. 
\begin{figure}[h]
    \centering
    \includegraphics[scale=0.7]{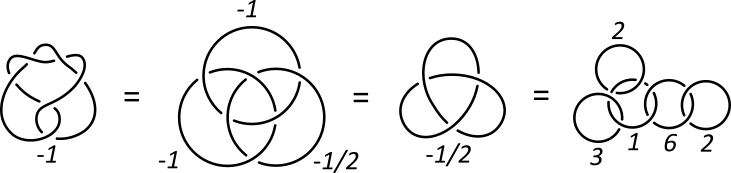}
    \caption{Different guises of $-1$-surgery on $m(\mathbf{5}_2)$}
    \label{fig:Sigma2311}
\end{figure}
It gives
\begin{align*}
\hat{Z}(S_{-1}^3(m(\mathbf{5}_2))) &= \hat{Z}(S^3_{-\frac{1}{2},-1,-1}(\mathbf{Bor})) = \hat{Z}(S^{3}_{-\frac{1}{2}}(m(\mathbf{3}_1))) = \hat{Z}(\Sigma(2,3,11))\\
&= q^{-\frac{3}{2}}\qty(1-q-q^9+q^{14}-q^{19}+q^{26}+q^{50}-O(q^{61})).
\end{align*}
This is consistent with the computation of $\hat{Z}$ using the formula for plumbed manifolds \cite{GPPV, GM}. 
In a similar manner, we find perfect agreement between $\hat{Z}$'s of $-2$ and $-3$-surgery on $m(\mathbf{5}_2)$ computed from the surgery formula and those computed from their plumbing description. 
The result is summarized in Table \ref{tab:52exceptionalsurgery}.
\begin{table}
    \centering
    \begin{tabular}{c c c}
        \hline\hline
        \multicolumn{2}{c}{$Y=S_p^3(m(\mathbf{5}_2))$} & $\hat{Z}_b(Y;q)$ \\
        \hline\hline
        $p=-1$ & $\Sigma(2,3,11) = M(-2;\frac{1}{2},\frac{2}{3},\frac{9}{11})$ & $q^{-\frac{3}{2}}\qty(1-q-q^9+q^{14}-q^{19}+q^{26}+q^{50}-O(q^{61}))$\\
        $p=-2$ & $M(-2;\frac{1}{2},\frac{3}{4},\frac{5}{7})$ & \begin{tabular}{@{}c@{}}$q^{-\frac{5}{4}}\qty(1 -q +q^{18} -q^{25} +q^{31} -q^{40}+O(q^{91}))$ \\ $-q^{\frac{5}{4}}\qty(1 -q^3 +q^6 -q^{11}+q^{45}-O(q^{56}))$\end{tabular}\\
        $p=-3$ & $M(-2;\frac{2}{3},\frac{2}{3},\frac{3}{5})$ & \begin{tabular}{@{}c@{}}$q^{-1}(1-q+q^8-q^{13}+q^{17} -q^{24}+O(q^{45}))$ \\ $-q^{\frac{4}{3}}(1-q^3+q^{27}-q^{36}+O(q^{84}))$ \end{tabular}\\
        \hline\hline
    \end{tabular}
    \caption{Some exceptional surgeries on $m(\mathbf{5}_2)$ and its $\hat{Z}$}
    \label{tab:52exceptionalsurgery}
\end{table}

Of course most surgery coefficients are not exceptional and they give hyperbolic manifolds. We list the first few $-\frac{1}{r}$-surgeries on $m(\mathbf{5}_2)$ in Table \ref{tab:52hyperbolicsurgery}.\footnote{For $-\frac{1}{r}$-surgeries with $r>0$, we have set $\epsilon q^d = q^{-\frac{r+r^{-1}}{4}}$, as it seems to be consistent with all examples we know.} 
\begin{table}
    \centering
    \begin{tabular}{c c}
        \hline\hline
        $Y=S_{-\frac{1}{r}}^3(m(\mathbf{5}_2))$ & $\hat{Z}(Y;q)$ \\
        \hline\hline
        $r=2$ & $q^{-\frac{3}{2}}\qty(1-2q +q^2 +2q^3 -2q^4 -q^5 -q^6 +3q^7 +2q^8 -2q^9 -O(q^{11}))$\\
        $r=3$ & $q^{-\frac{3}{2}}\qty(1 -2q +q^2 +q^3 -q^4 +q^5 -2q^6 +2q^9 +3q^{10} -3q^{11} -O(q^{12}))$\\
        $r=4$ & $q^{-\frac{3}{2}}\qty(1 -2q +q^2 +q^3 -q^4 -q^6 +2q^7 -q^8 -q^9 +q^{10} +q^{11} +O(q^{12}))$\\
        $r=5$ & $q^{-\frac{3}{2}}\qty(1 -2q +q^2 +q^3 -q^4 -q^6 +q^7 +q^9 -2q^{11} +2q^{13} +O(q^{14}))$\\
        \hline\hline
    \end{tabular}
    \caption{Some hyperbolic surgeries on $m(\mathbf{5}_2)$ and its $\hat{Z}$}
    \label{tab:52hyperbolicsurgery}
\end{table}


\subsubsection{$-\frac{1}{r}$-surgeries on double twist knots}
Another neat class of examples where we can check the consistency of the surgery formula comes from $-\frac{1}{r}$-surgeries on double twist knots. Observe that the double twist knot $K_{m,n}$ can be viewed as $-\frac{1}{m},-\frac{1}{n}$-surgery on two components of the Borromean rings (Figure \ref{fig:doubletwistfromBorromean}). 
\begin{figure}[h]
    \centering
    \includegraphics[scale=0.5]{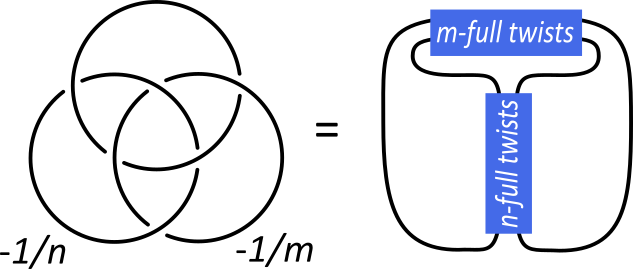}
    \caption{Double twist knots as $-\frac{1}{n}, -\frac{1}{m}$-surgery on two components of the Borromean rings}
    \label{fig:doubletwistfromBorromean}
\end{figure}
Because the three components of the Borromean rings are symmetric, it follows that
\[S_{-\frac{1}{r}}^3(K_{m,n}) = S^3_{-\frac{1}{m},-\frac{1}{n},-\frac{1}{r}}(\mathbf{Bor}) = S_{-\frac{1}{n}}^3(K_{m,r}).\]
Using the surgery formula, one can readily check that the identity holds at the level of $\hat{Z}$. For instance, the first few terms of $F_{K_{2,2}=m(\mathbf{7}_4)}^+(x,q)$ are
\begin{align*}
    f_0^{m(\mathbf{7}_4)} &= -q^{-1}+2-q-2q^2+2q^3+q^5-2q^6+2q^8-2q^9+2q^{10}-q^{11}-2q^{12}+O(q^{14}),\\
    f_1^{m(\mathbf{7}_4)} &= -q^{-1}+2-q-2q^2+3q^3-q^5-4q^6+q^7+6q^8+q^9+2q^{10}-5q^{11}-8q^{12}-O(q^{13}),\\
    f_2^{m(\mathbf{7}_4)} &= -q^{-1}+2-q-q^2+3q^3-q^4-3q^5-5q^6+3q^7+10q^8+5q^9+O(q^{10}),\\
    f_3^{m(\mathbf{7}_4)} &= -q^{-1}+2-q^2+2q^3-3q^4-3q^5-3q^6+5q^7+O(q^8),
\end{align*}
and the $-1$-surgery gives the same result as $\hat{Z}(S^3_{-\frac{1}{2}}(K_{2,1}=m(\mathbf{5}_2));q)$ in Table \ref{tab:52hyperbolicsurgery}.

\subsection{Reverse-engineering $F_K$ for some simple links}
\subsubsection{Partial surgery formula}
The surgery formula in Conjecture \ref{conj:linksurg} is about closed 3-manifolds, but sometimes it is useful to do `partial surgery', i.e.\, Dehn surgery on some components of a link while leaving the other components unfilled (Figure \ref{fig:partialsurgery}). In this way, we can relate $F_K$'s for different links, when one is obtained by another by partial surgery. 
Any full surgery is a sequence of partial surgeries, and regardless of the order of partial surgeries, we should get the same answer. This consistency condition determines the form of partial surgery formula, and in this sense, the following is a corollary of Conjecture \ref{conj:linksurg}: 
\begin{figure}[h]
    \centering
    \includegraphics[scale=0.6]{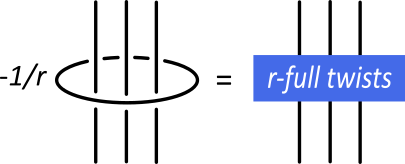}
    \caption{Partial surgery (a Rolfsen move in this case)}
    \label{fig:partialsurgery}
\end{figure}
\begin{conj}
Suppose $L$ is a link with components $L_0, L_1, \cdots, L_l$. Assume that $L_0$ is an unknot. Then doing $-\frac{1}{r}$-surgery on $L_0$, we get\footnote{Here, $\cong$ denotes equality up to $\epsilon q^d$. That is, $f\cong g$ iff $f = \epsilon q^d g$ for some sign $\epsilon \in \{\pm 1\}$ and $d\in \mathbb{Q}$.}
\begin{equation}\label{eq:partialsurg}
    F_{L'}(x_1,\cdots,x_l,q) \cong \mathcal{L}\qty[(x_0^{\frac{1}{2r}}-x_0^{-\frac{1}{2r}})F_L(x_0,\cdots,x_l,q)]
\end{equation}
whenever the r.h.s.\ makes sense, where
\begin{equation}
    \mathcal{L}: x_0^u \mapsto q^{ru^2}\prod_{1\leq i\leq l} x_i^{r\,lk(L_i,L_0) u}.
\end{equation}
\end{conj}
There is a reason we don't consider partial surgeries of the form $-\frac{p}{r}$ with $p>1$. It is because we want each strand of the link to be null-homologous. When a component of the link represents a non-trivial cycle in the homology of the ambient $3$-manifold, then we can no longer have $x$ as a variable independent from $q$, as it gets mixed up with the $\mathrm{Spin}^c$-variables $b$. Of course, if $L_0$, the link component where we do surgery, is algebraically split from the others, then $x$ and the $\mathrm{Spin}^c$-variables $b$ do not interact, and we can use the same formula as in Conjecture \ref{conj:surgery}.

\subsubsection{Torus links}
The partial surgery formula is useful in computing $F_K$ for links. In order to illustrate this, we give some examples of torus links.  Recall that for any tree $\Gamma$, there is a naturally associated link $L_\Gamma$; Figure \ref{fig:treelink}. 
\begin{figure}[h]
    \centering
    \includegraphics[scale=0.8]{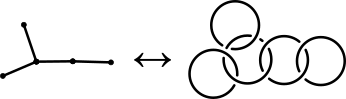}
    \caption{A link naturally associated to a tree}
    \label{fig:treelink}
\end{figure}
Let's call such links \emph{tree links}. For tree links, we know that\footnote{When $\deg v>1$, it should be understood as the power series expansion in $x$.}
\begin{equation}
    F_{L_\Gamma}(x_1,\cdots,x_l,q) \cong \prod_{v\in V}\qty(x_v^{1/2}-x_v^{-1/2})^{1-\deg v}.
\end{equation}
Starting from a tree link and doing some $-\frac{1}{r}$ partial surgeries, we get $F_L$ for various torus links.
For example, starting from a tree with vertices $v_0, \cdots, v_p$ where $v_0$ is connected to all the other vertices, we can do $-\frac{1}{r}$ surgery on the central vertex. The resulting link is the $T_{rp,p}$-torus link, and its has
\begin{equation}
    F_{T_{rp,p}}(x_1,\cdots,x_p,q) \cong \mathcal{L}\qty[\frac{x_0^{\frac{1}{2r}}-x_0^{-\frac{1}{2r}}}{(x_0^{\frac{1}{2}}-x_0^{-\frac{1}{2}})^{p-1}}],
\end{equation}
where 
\[\mathcal{L}: x_0^u \mapsto q^{ru^2}\prod_{1\leq i\leq p}x_i^{ru}.\]
In particular, 
\begin{equation}
    F_{T_{4,2}}^+(x,y,q) \cong x^{\frac{1}{2}}y^{\frac{1}{2}}\sum_{m\geq 0}(-1)^mq^{\frac{m(m+1)}{2}}x^my^m.
\end{equation}
Further $-\frac{1}{r}$-surgery on a component of $T_{4,2}$ gives $T_{2r+1,2}$, and in this way we recover the previously known result on $F_{T_{2r+1,2}}(x,q)$ \cite{GM}.

\subsubsection{Whitehead link and Borromean rings}
Some more interesting examples are the Whitehead link and the Borromean rings. We have already seen how the Borromean rings are closely related to double twist knots; $K_{m,n}$ is the $-\frac{1}{m}, -\frac{1}{n}$-surgery on two components of the Borromean rings. Whitehead link is also a close cousin, as it is the $-1$-surgery on a component of the Borromean rings (Figure \ref{fig:Whiteheadlinks}). 
\begin{figure}[h]
    \centering
    \includegraphics[scale=0.5]{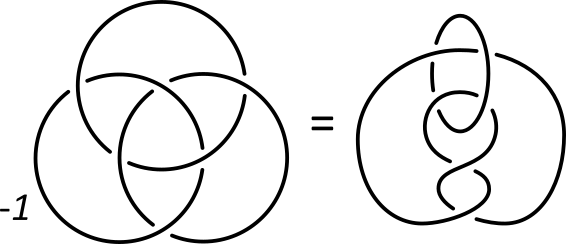}
    \caption{Whitehead link as $-1$ surgery on a component of the Borromean rings}
    \label{fig:Whiteheadlinks}
\end{figure}
Hence the $-\frac{1}{r}$-surgery on a component of the Whitehead link is the twist knot $K_{r,1}$. 
Knowing how to compute $F_{K_{m,n}}(x,q)$ for all $m,n>0$ big enough, we can deduce $F_K$ for the Whitehead link and the Borromean rings. This is because, in the partial surgery formula \eqref{eq:partialsurg}, the contribution of each coefficient of $F_L$ as a series in $x_0$ gets far apart from one another as $r$ gets large. In another words, if we know $F_K$ of the $-\frac{1}{r}$-surgery on a component of $L$ for all sufficiently large $r$, we can uniquely determine $F_L$. We refer to this process of deducing $F_K$ from its $-\frac{1}{r}$-surgeries as \emph{reverse engineering}. 

Reverse-engineering $F_K$ for the Whitehead link from $F_K$ for twist knots, we get
\[F_{\mathbf{Wh}}^+(x,y,q) \cong x^{\frac{1}{2}}y^{\frac{1}{2}}\sum_{j\geq 0}f_j^{\mathbf{Wh}}(x,q)y^{j} = x^{\frac{1}{2}}y^{\frac{1}{2}}\sum_{i,j\geq 0}f_{i,j}^{\mathbf{Wh}}(q)x^iy^j.\]
where $f_{i,j}^{\mathbf{Wh}}(q) = f_{j,i}^{\mathbf{Wh}}(q)$ because the Whitehead link is symmetric, and the first few coefficients are given by
\begin{align*}
    f_0^{\mathbf{Wh}}(x,q) &= 1 + x +x^2 +x^3 +x^4 +x^5 +\cdots = \frac{1}{1-x},\\
    f_1^{\mathbf{Wh}}(x,q) &= 1 +(-q^{-1}+1+q)x +(-q^{-2}-q^{-1}+1+q+q^2)x^2\\ &\quad+(-q^{-3}-q^{-2}-q^{-1}+1+q+q^2+q^3)x^3\\
    &\quad+(-q^{-4}-q^{-3}-q^{-2}-q^{-1}+1+q+q^2+q^3+q^4)x^4 +O(x^5)\\
    &= \sum_{i\geq 0}\frac{q^{i+1}+q^{-i-1}-2}{q-1}x^i,\\
    f_2^{\mathbf{Wh}}(x,q)  &= 1 +(-q^{-2}-q^{-1}+1+q+q^2)x +(-2q^{-2}-2q^{-1}+q+2q^2+q^3+q^4)x^2\\
    &\quad+(q^{-5}-q^{-3}-3q^{-2}-3q^{-1}-1+2q^2+2q^3+2q^4+q^5+q^6)x^3\\
    &\quad+(q^{-7}+q^{-6}+q^{-5}-q^{-4}-2q^{-3}-4q^{-2}-4q^{-1}-2-q+q^2+2q^3+3q^4+2q^5\\
    &\quad\quad+2q^6+q^7+q^8)x^4 +O(x^5),
\end{align*}
\begin{align*}
    f_3^{\mathbf{Wh}}(x,q) &= 1 +(-q^{-3}-q^{-2}-q^{-1}+1+q+q^2+q^3)x\\
    &\quad+(q^{-5}-q^{-3}-3q^{-2}-3q^{-1}-1+2q^2+2q^3+2q^4+q^5+q^6)x^2 +O(x^3),\\
    f_4^{\mathbf{Wh}}(x,q) &= 1+(-q^{-4}-q^{-3}-q^{-2}-q^{-1}+1+q+q^2+q^3+q^4)x\\
    &\quad+(q^{-7}+q^{-6}+q^{-5}-q^{-4}-2q^{-3}-4q^{-2}-4q^{-1}-2-q+q^2+2q^3+3q^4+2q^5\\
    &\quad\quad+2q^6+q^7+q^8)x^2 +O(x^3).
\end{align*}
In other words, 
\begin{align*}
    f_{i,0}^{\mathbf{Wh}}(q) &= 1 \quad\text{for all }i,\\
    f_{i,1}^{\mathbf{Wh}}(q) &= \frac{q^{i+1}+q^{-i}-2}{q-1} \quad\text{for all }i,\\
    f_{i,2}^{\mathbf{Wh}}(q) &= 1+\sum_{k=0}^{i}\frac{q^{2k+1}+(q^{-k}-q^k)(1+q^{-1})-q^{-2k+1}}{q-1} \quad\text{for all }i,\\
    f_{3,3}^{\mathbf{Wh}}(q) &= q^{-7}+q^{-6}+2q^{-5}+q^{-4}-2q^{-3}-4q^{-2}-6q^{-1}-4-3q+2q^3+3q^4+3q^5+3q^6\\
    &\quad\quad+2q^7+q^8+q^9,
\end{align*}
and so on. 
Every coefficient seems to be a Laurent polynomial in $q$. 
In the classical limit $q\rightarrow 1$, we get
\[F_{\mathbf{Wh}}(x,y,1) = \frac{1}{(x^{\frac{1}{2}}-x^{-\frac{1}{2}})(y^{\frac{1}{2}}-y^{-\frac{1}{2}})}\]
which is, as expected, the inverse of the Alexander-Conway function for the Whitehead link.
It is pleasant to observe that we can recover the result on $F_K$ for the figure-eight knot given in \cite{GM} by doing the $1$-surgery on a component of the Whitehead link. Similarly, after orientation reversal and doing $-\frac{1}{r}$-surgery on one of the components, we get $F_K$ for twist knots $K_{r,-1}$ with $r\geq 1$. 

By looking at $F_{K_{m,n}}$ for fixed $m$, we can also reverse-engineer for the $-\frac{1}{m}$ surgery on a component of the Borromean rings. For instance, when $m=2$, we have
\begin{align*}
    f^{\mathbf{Bor}_{-1/2}}_{i,0}(q) &= 1-q+q^3-q^6+q^{10}-q^{15}+q^{21}-q^{28}+O(q^{36}) \quad\text{for all }i,\\
    f^{\mathbf{Bor}_{-1/2}}_{1,1}(q) &= 1-2q+3q^3+q^4-q^5-3q^6-q^7+q^9+O(q^{10}),\\
    f^{\mathbf{Bor}_{-1/2}}_{2,1}(q) &= -2q+q^2+4q^3+4q^4-3q^5+O(q^6).
\end{align*}
Once we have $F_K$ for the $-\frac{1}{m}$-surgery on the Borromean rings for all $m$, we can use that to reverse-engineer $F_K$ for the Borromean rings. Let's write
\[F_{\mathbf{Bor}}^+(x,y,z,q) \cong x^{\frac{1}{2}}y^{\frac{1}{2}}z^{\frac{1}{2}}\sum_{i,j,k\geq 0}f_{i,j,k}^{\mathbf{Bor}}(q)x^iy^jz^k,\]
where $f_{i,j,k}^{\mathbf{Bor}}$ is symmetric under permutation of the indices. 
Reverse-engineering, one can deduce that 
\begin{align*}
    f_{i,j,0}^{\mathbf{Bor}}(q) &= -1\quad\text{for all }i,j\geq 0.
\end{align*}
This is consistent with the fact that in the classical limit, we should have
\[F_{\mathbf{Bor}}(x,y,z,1) = \frac{1}{(x^{\frac{1}{2}}-x^{-\frac{1}{2}})(y^{\frac{1}{2}}-y^{-\frac{1}{2}})(z^{\frac{1}{2}}-z^{-\frac{1}{2}})}.\]

\subsubsection{An approach to general knots}
Before ending this section, let us present an idea toward $F_K$ for general knots. As we mentioned in the introduction, the problem of mathematically defining the GPPV series $\hat{Z}$ can be viewed as a two-step problem, first defining $F_K$ for links, and then making a surgery formula that works for all surgery coefficients. In Section \ref{sec:main} we have made some progress toward the first step by giving a definition for positive braid knots. Positive braid knots may seem to be too special, but when combined with the partial surgery formula, they are actually quite general, in a sense we are about to describe. 

Let $K$ be any knot, presented as the closure of a braid $\beta$. Let's say $\beta$ has $n$ number of negative crossings, $\sigma_{i_1}^{-1},\cdots,\sigma_{i_n}^{-1}$. Replacing each negative crossing $\sigma_{i_k}^{-1}$ with an odd power of positive crossing $\sigma_{i_k}^{2r_k+1}$, we get a positive braid parametrized by non-negative integers $r_1,\cdots,r_n$. Let's call its braid closure $K(r_1,\cdots,r_n)$. We know how to compute $F_{K(r_1,\cdots,r_n)}(x,q)$ for any non-negative integers $r_1,\cdots,r_n$, and these are sufficient data to reverse-engineer $F_L(x,y_1,\cdots,y_n,q)$ where $L$ is the link obtained by replacing each negative crossings of $K$ by a positive crossing with an auxiliary unknot linked to it, as in Figure \ref{fig:Amatrix}.\footnote{We can study odd powers of the flipped $R$-matrix to reverse-engineer the matrix associated to the tangle in Figure \ref{fig:Amatrix}. It will be a linear map $V_\infty\otimes V_\infty \rightarrow V_\infty\otimes V_\infty$ over the field $\Bbbk_{x,y} := \mathbb{C}(q^{\frac{1}{2}},x^{\frac{1}{2}},y^{\frac{1}{2}})$.}
\begin{figure}[h]
    \centering
    \includegraphics[scale=0.5]{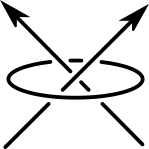}
    \caption{A crossing with an auxiliary unknot linked to it}
    \label{fig:Amatrix}
\end{figure}
This is because $K(r_1,\cdots,r_n)$ is the $-\frac{1}{r_1},\cdots,-\frac{1}{r_n}$-surgery on the auxiliary components of $L$. Having $F_L(x,y_1,\cdots,y_n,q)$ doesn't seem to be too far from having $F_K(x,q)$, as $K$ is $1,\cdots,1$-surgery on the auxiliary components of $L$. What is missing is a surgery formula that works for any surgery coefficients. In this sense, the problem of defining $F_K$ can be reduced to the problem of finding a general surgery formula.

\section{Strange identity for positive braid knots}\label{sec:strangeid}
Throughout this article, our main focus was to understand $F_K$ better. In this section, we turn our attention to a neat corollary of Theorem \ref{pbthm}. According to the theorem, for a positive braid $\beta$ and its closure $K$, 
\[\Tr_q' \bm{\beta}_{V_\infty^h} =
\frac{F_K^-(x,q)}{x^{\frac{1}{2}}-x^{-\frac{1}{2}}}.\]
As we have emphasized, If we set $q=e^\hbar$, their $\hbar$-expansion agrees with the large color expansion of the colored Jones polynomial. Thanks to Weyl symmetry, we see that
\[\Tr_q' \bm{\beta}_{V_\infty^h} \text{``}=\text{''} \frac{F_K^{\pm}(x,q)}{x^{\frac{1}{2}}-x^{-\frac{1}{2}}}\]
at the level of $\hbar$-expansion. 
Now let's specialize both sides by $x=1$ to get the ``identity'': 
\begin{equation}
    \Tr_q' \bm{\beta}_{V_\infty^h} \bigg\vert_{x=1} \text{``}=\text{''} \frac{F_K^{\pm}(x,q)}{x^{\frac{1}{2}}-x^{-\frac{1}{2}}}\bigg\vert_{x=1}.
\end{equation}
Note that the l.h.s.\ and its derivatives only makes sense when $q$ is a root of unity (and agrees with the Kashaev's invariant), and the r.h.s. is a well-defined $q$-series that makes sense inside the unit disk. The identity holds in a sense of perturbative series in $\hbar$, near $q=1$. 
Thus, as stated in Corollary \ref{strangeidentity}, we found a strange identity for every positive braid knot! 

For example, when $K$ is the right-handed trefoil, 
\[\langle \mathbf{3}_1 \rangle \text{``}=\text{''} \frac{F_{\mathbf{3}_1}^{\pm}(x,q)}{x^{\frac{1}{2}}-x^{-\frac{1}{2}}} = -\frac{q}{2}\sum_{m=1}^{\infty}m\qty(\frac{12}{m})q^{\frac{m^2-1}{24}} = \frac{1}2{\qty(-q+5q^2+7q^3-11q^6 -\cdots)}.\]
This is the famous Kontsevich-Zagier strange identity \cite{Z}. 
For another example, take $\mathbf{10}_{139}$, which is non-torus. Then we have
\[\langle \mathbf{10}_{139} \rangle \text{``}=\text{''} \frac{F_{\mathbf{10}_{139}}^{\pm}(x,q)}{x^{\frac{1}{2}}-x^{-\frac{1}{2}}} = \frac{1}{2}\qty(-7q^4 +26q^6 -15q^7 +17q^8 +\cdots).\]
It is a very interesting problem to see if this ``strange identity'' holds near other roots of unity.

\section{Open questions and future directions}
This paper opens up new directions of study and leave some interesting open questions. We conclude this paper by listing some of those. 

\textbf{Future directions: }
\begin{itemize}
    \item Extend this $R$-matrix approach to find a full mathematical definition of $F_K$.
    \item Understand the geometrical meaning of this $R$-matrix approach, especially the apparent connection with contact topology. 
    \item It is conjectured that $\hat{Z}$ enjoys quantum modularity. Study quantum modular properties of $F_K$. 
    \item Find a surgery formula that works for any surgery coefficients. 
    \item Ultimately, categorify $F_K$ and $\hat{Z}$. 
\end{itemize}

\textbf{Some more specific questions: }
\begin{itemize}
    \item When the coefficients of $F_K$ are infinite $q$-series (e.g.\ for non-fibered knots with non-monic Alexander polynomial), is there an easy way to find $F_K$ for the mirror knot?
    \item Find closed form expressions for $F_{\mathbf{Wh}}$ and $F_{\mathbf{Bor}}$. 
    \item We expect that Theorem \ref{pbthm} can be generalized to higher rank. It would be nice to work it out explicitly. 
\end{itemize}

\bibliography{FKexamples}
\bibliographystyle{alpha}

\end{document}